\documentclass[12pt]{elsarticle}
\usepackage{amsfonts,amssymb,amsmath,mathrsfs}
\usepackage{lineno,hyperref}
\modulolinenumbers[5]

\usepackage{color}

\makeatletter
\def\ps@pprintTitle{%
	\let\@oddhead\@empty
	\let\@evenhead\@empty
	\let\@oddfoot\@empty
	\let\@evenfoot\@oddfoot
}
\makeatother

\newtheorem{thm}{Theorem}[section]
\newtheorem{lem}[thm]{Lemma}

\newtheorem{rem}[thm]{Remark}

\newtheorem{defn}[thm]{Definition}
\numberwithin{equation}{section}

\newenvironment{proof}[1][\noindent \textbf{Proof: }]{#1}{ \hfill $\square$ \vspace{2mm}}

\begin{document}

	\begin{frontmatter}
		
		\title{Global ultradifferentiable   hypoellipticity on compact manifolds}

		\author[addressUFPR]
		{Fernando de \'Avila Silva\corref{correspondingauthor}}
		\cortext[correspondingauthor]{Corresponding author}
		\ead{fernando.avila@ufpr.br}
		
				\address[addressUFPR]{Department of Mathematics, Federal University of Paran\'a, Caixa Postal 19081, \\ CEP 81531-980, Curitiba, Brazil}
		
		\author[addressUFPR]{Eliakim Cleyton Machado} 
		\ead{eliakimmachado@gmail.com}

		\begin{abstract}
		We study the  global hypoellipticity problem for certain linear operators in Komatsu classes of Roumieu and Beurling type on compact manifolds. We present an approach by combining a characterization of these spaces via eigenfuction expansions, generated by an elliptic operator, and the analysis of matrix-symbols obtained by these  expansions.
		\end{abstract}
		

	\end{frontmatter}

\section{Introduction \label{intro}}

We are interested in the study of global hypoellipticity in the setting of ultradifferentiable  classes for certain linear operators on a compact manifold $X$. Broadly speaking, we aim to analyze the following problem: Let $\mathscr{F}(X)$ be an ultradifferentiable class of functions on $X$, $\mathscr{U}(X)$ its dual space and  $P: \mathscr{U}(X) \to \mathscr{U}(X)$  a linear operator. If $u \in \mathscr{U}(X)$ is such that $Pu \in \mathscr{F}(X)$, what conditions guarantee that $u \in \mathscr{F}(X)$?

One of the interests in the study of global hypoellipticity is the fact that local and global cases are rather different in general. For instance, there are classes of vector fields on the torus that are globally hypoelliptic despite not being locally hypoelliptic (see \cite{GW72}). Also, the  global properties are open problems, except for some particular classes of operators, that seem impossible to be solved by a unified approach. 

Some authors, however, have obtained significant advances by considering special cases. For instance, classes of differential operators in the ultradifferentiable setting  on the torus $\mathbb{T}^n$, e.g., \cite{AD2014,Victor_2021,IGB,YG93}, and
vector fields on compact Lie groups  as  presented by A. Kirilov, W. A. A. de Moraes and M. Ruzhansky in \cite{KWR}. Still on the topic of Lie groups, we quote the article \cite{Araujo2019}, by G. Ara\'ujo, directed to the study of global properties of a class of systems acting in some functional spaces such as analytic and Gevrey. In particular, his approach is based on a notion of invariance with respect to the Laplace-Beltrami operator.

Taking into account these facts, we attack the proposed problem  by taking inspirations from the following two works: first, the approach presented by A. Dasgupta and M. Ruzhanky in \cite{DR16}  characterizing Komatsu classes on compact manifolds in terms of a Fourier analysis generated by elliptic operators; in  second, we extend the techniques used by A. Kirilov  and  W. A. A. de Moraes in \cite{KM20} for the study of  global properties in the $C^{\infty}$ sense for certain classes of  invariant operators.

The characterizations  given in \cite{DR16} allows to represent elements in ultradifferentiable classes as expansions of type
\begin{equation*}
	u = \sum_{\ell \in \mathbb{N}_0} \left\langle \widehat{u}(\ell),e_\ell(x)\right\rangle_{\mathbb{C}^{d_\ell}}, 
\end{equation*}
given by a fixed elliptic operator $E$. On the other hand, motivated by \cite{KM20} (see also \cite{Araujo2019,GW73}), we use the fact that 
a linear $E$-invariant operator $P$, satisfying certain conditions,  has a matrix-symbol   $\sigma_P(\ell)\in \mathbb{C}^{d_\ell\times d_\ell}$, $\ell \in \mathbb{N}_0$, for which
\begin{equation*}
	Pu = \sum_{\ell \in \mathbb{N}_0} \left\langle\sigma_P(\ell)\widehat{u}(\ell),e_\ell(x)\right\rangle_{\mathbb{C}^{d_\ell}}.
\end{equation*}

Therefore, a solution of equation $Pu=f$ satisfy
\begin{equation*}
	\sigma_P(\ell)\widehat{u}(\ell) = \widehat{f}(\ell), \ \forall \ell \in \mathbb{N}_0,
\end{equation*} 
and consequently its regularity  can  be completely characterized in terms of the behavior of  $\sigma_P(\ell)$, as $\ell \to \infty$, in a suitable sense.

Our work is then organized as follows: Section \ref{sec2}
contains the ultradifferentiable spaces, as introduced in \cite{DR16}, 
and its characterizations in terms of Fourier coefficients. The classes of invariant operators are presented by  Definition \ref{def-invariant}.
Section \ref{sec3} discuss the global hypoellipticity in the Roumieu case, as it shall be stated in Theorem \ref{theorem-GH-R}, necessary and sufficient conditions  are presented in terms of matrix-symbols.
Section \ref{sec4}  extends these conditions to the Beurling case in Theorem \ref{theorem-GH-B}.

\section{Preliminaries  \label{sec2}}

Let $X$ be a n-dimensional, closed, smooth manifold endowed with a positive measure $dx$. The inner product on the Hilbert space
$L^2(X) = L^2(X,dx)$ is given by
\begin{equation*}
	(f,g)_{L^2(X)} \doteq \int_{X}{f(x)\overline{g(x)}dx}.
\end{equation*}

By $\Psi^\nu(X)$ we denote the class of classical, positive, elliptic, pseudo-differential operators, of order $\nu\in \mathbb{N}$. It is well know that if $E \in \Psi^\nu(M)$, then  its spectrum  is a discrete subset of the real line and coincides with the set of all its eigenvalues. Moreover, the eigenvalues of $E$ 
form a sequence
$$
0=\lambda_0 < \lambda_1 < \ldots < \lambda_\ell \longrightarrow \infty,
$$
counting the multiplicity.

For each $\ell \in \mathbb{N}_0 \doteq \mathbb{N} \cup \{0\}$,  the corresponding eigenspace associated with $\lambda_\ell$, which is finite dimensional, is denoted by $\mathscr{H}_\ell\doteq\ker(E-\lambda_\ell I)$. In particular,  we may fix an orthornormal basis 
$\{e_{\ell,k}\}_{k=1}^{d_\ell}$ on $\mathscr{H}_\ell$ such that 
$$
L^2(X)=\widehat{\bigoplus_{\ell \in \mathbb{N}_0}}\mathscr{H}_\ell,
$$ 
where $d_\ell \doteq dim(\mathscr{H}_\ell)$. Therefore, 
for every $f\in L^2(X)$ we have 
$$
f(x)=\sum_{\ell=0}^\infty\sum_{k=1}^{d_\ell}\widehat{f}(\ell,k) e_{\ell,k}(x),
$$
with	$\widehat{f}(\ell,k) = 	(f,e_{\ell,k})_{L^2(X)}$, for  all $\ell \in \mathbb{N}_0$ and  $k = 1, \ldots, d_{\ell}.$

Given $\ell \in \mathbb{N}_0$, the vector
$$
\widehat{f}(\ell) \doteq \left(\begin{array}{c}
	\widehat{f}(\ell,1) \\ 
	\vdots \\ 
	\widehat{f}(\ell,d_\ell)
\end{array} \right)\in \mathbb{C}^{d_\ell}
$$
is said to be a \textit{Fourier coefficient} of $f \in L^{2}(X)$. Its Hilbert Schmidt norm is
$$
\|\widehat{f}(\ell)\|_\mathtt{HS}\doteq\left(\sum_{k=1}^{d_\ell}|\widehat{f}(\ell,k)|^2\right)^{1/2}
$$
and, by  Plancherel formula, we get
$$
\|f\|_{L^2(X)}^2=\sum_{\ell=0}^\infty\sum_{k=1}^{d_\ell}|\widehat{f}(\ell,k)|^2=\sum_{\ell=0}^\infty\|\widehat{f}(\ell)\|_\mathtt{HS}^2.
$$
Moreover, $f \in C^\infty(X)$ if and only if, given $N>0$ there is $C_N>0$ such that
\begin{equation}\label{smooth-charac}
	\|\widehat{f}(\ell)\|_\mathtt{HS}\leq C_N(1+\lambda_\ell)^{-N}, \ \forall \ell \in \mathbb{N}.
\end{equation}

\subsection{Ultradifferentiable classes  \label{sec2.1}}

We now  give the definitions of  ultradifferentiable classes on $X$ as  introduce in \cite{DR16}. For this, we denote by $\mathscr{M}\doteq\{M_k\}_{k\in \mathbb{N}_0}$  a sequence of real numbers satisfying the following: There are constants 	$H>0$ and $A\geq1$  such that
\begin{description}
	\item[(M.0)] $M_0=M_1=1$,
	\item[(M.1)] $M_{k+1}\leq AH^kM_k, \ \forall  k\in \mathbb{N}_0$,
	\item[(M.2)] $M_{2k}\leq AH^{2k}M_k^2, \ \forall k\in \mathbb{N}_0$,
	\item[(M.3)] $M_k^2\leq M_{k-1}M_{k+1}, \ \forall  k\in \mathbb{N}$.
\end{description}

Also, for the Roumieu (Beurling) case we assume that there are constants $L>0$ and $C>0$ (for  every $L>0$, there exists $C>0$) such that
$$
k!\leq CL^kM_k, \ \forall k\in\mathbb{N}_0.
$$

Associated to  $\mathscr{M}$ we define a function 
$\mathcal{M}: [0,\infty)\to \mathbb{R}$ as follows: 
$$
\mathcal{M}(0)\doteq0  \ \textrm{ and } \ \mathcal{M}(r)\doteq\sup_{k\in\mathbb{N}}\left[\log\left(\frac{r^{\nu k}}{M_{\nu k}}\right)\right],\quad\forall r>0.
$$

Function  $\mathcal{M}$ is non-decreasing and, in view of  the proof of Theorem 4.2. in \cite{DR16}, satisfy  the following property.
\begin{lem}\label{lemma_several_estimates}
	Given 	 $\ell \in \mathbb{N}$ and  $L>0$ we have
	$$
	\exp \left( -\frac{1}{2}\mathcal{M}(L\lambda_\ell^{1/\nu})\right) \leq
	\exp \left(-\mathcal{M}(\widetilde{L}\lambda_\ell^{1/\nu})\right), \ \text{for}\ \widetilde{L}\doteq\frac{L}{\sqrt{A}H}.
	$$		
\end{lem}

We are now in condition to introduce the ultradifferentiable classes.
\begin{defn}
	Let $E \in \Psi^\nu(X)$ be fixed. We denote by $\mathscr{E}_{\mathscr{M}}(X)$ ($\mathscr{E}_{(\mathscr{M})}(X)$) the Roumieu (Beurling) space of functions $\phi\in C^\infty(X)$ such that there exists 
	$h>0$ and $C>0$ (for all $h>0$ there exists $C_h>0$)  satisfying 
	\begin{equation*}
		\|E^k\phi\|_{L^2(X)}\leq Ch^{\nu k}M_{\nu k},\quad\forall k\in\mathbb{N}_0.
	\end{equation*}
	\begin{equation*}
		\left(\|E^k\phi\|_{L^2(X)}\leq C_hh^{\nu k}M_{\nu k},\quad\forall k\in\mathbb{N}_0.\right)
	\end{equation*}
	
\end{defn}

\begin{rem}
	It follows from Theorem 2.3 in \cite{DR16} that these spaces
	are independent of a particular choice of $E \in \Psi^\nu(X)$, justifying the notations $\mathscr{E}_{\mathscr{M}}(X)$ and $\mathscr{E}_{(\mathscr{M})}(X)$.
\end{rem}

The symbols $\mathscr{E}_{\mathscr{M}}'(X)$ and $\mathscr{E}_{(\mathscr{M})}'(X)$ denotes the   space of ultradistributions of Roumieu and  Beurling type, respectively. The  \textit{Fourier coefficients} of an ultradistribution $u$  are defined by expressions
$$
\widehat{u}(\ell) \doteq 
\left(\begin{array}{c}
	\widehat{u}(\ell,1) \\ 
	\vdots \\ 
	\widehat{u}(\ell,d_\ell)
\end{array}\right), \ 
\widehat{u}(\ell,k) \doteq u(\overline{e_{\ell,k}}), \ k=1, \ldots, d_{\ell}.
$$

The next results characterize the ultradifferentiable classes in terms of the Fourier coefficients and the eigenvalues of the operator $E$.

\begin{thm}\label{T_charac_ultrad}
	A smooth function $\phi$ on $X$ belongs to  $\mathscr{E}_{\mathscr{M}}(X)$ ($ \mathscr{E}_{(\mathscr{M})}(X)$) if and only if there are positive constants $C$ and $L$ (for every $L>0$ there is $C_{L}>0$) such that
	\begin{equation*}
		\|\widehat{\phi}(\ell)\|_{\mathtt{HS}}\leq C\exp\left[-\mathcal{M}(L\lambda_\ell^{1/\nu})\right], \ \forall\ell \in \mathbb{N}_0,
	\end{equation*}
	\begin{equation*}
		\left(\|\widehat{\phi}(\ell)\|_{\mathtt{HS}}\leq C_L\exp\left[-\mathcal{M}(L\lambda_\ell^{1/\nu})\right], \ \forall \ell \in \mathbb{N}_0.\right)	
	\end{equation*}
	
\end{thm}

\begin{thm}\label{T_charac_ultradistribution}
	A linear functional $u: \mathscr{E}_{\mathscr{M}}(X) \to \mathbb{C}$ ($u: \mathscr{E}_{(\mathscr{M})}(X) \to \mathbb{C}$) belongs to 
	$\mathscr{E}_{\mathscr{M}}'(X)$ 
	($ \mathscr{E}_{(\mathscr{M})}'(X)$) if and only if for every $L >0$ there exists $K_L>0$ (there exists $K>0$ and $L>0$) such that 
	\begin{equation*}
		\|\widehat{u}(\ell)\|_\mathtt{HS} \leq K_L\exp\left[\mathcal{M}(L\lambda_\ell^{1/\nu})\right], \ \forall \ell \in \mathbb{N}_0.
	\end{equation*}
	\begin{equation*}
		\left(\|\widehat{u}(\ell)\|_\mathtt{HS} \leq K\exp\left[\mathcal{M}(L\lambda_\ell^{1/\nu})\right], \ \forall \ell \in \mathbb{N}_0.\right)
	\end{equation*}

\end{thm}

\subsection{Invariant operators on ultradifferentiable classes \label{sec2.2}}

We now introduce the classes of operators under investigations. 
We emphasize that at this point we are taking inspiration from Proposition 2.1 and Definition 2.2 in \cite{KM20}.

\begin{defn}\label{def-invariant}
	Let $P: C^\infty(X)\to C^\infty(X)$ be a linear operator with continuous extension to $\mathscr{D}'(X)$ and assume that  the domain of $P^\ast$, the adjoint of $P$, contains $C^\infty(X)$.  Given   $E\in \Psi^\nu(X)$, we say that $P$ is:
	
	\begin{enumerate}
		\item [1.] 	$C^\infty$-invariant with respect to $E$ if 
		$$
		P(\mathscr{H}_\ell)\subseteq\mathscr{H}_ \ell, \ \forall \ell \in \mathbb{N}_0.
		$$
		
		\item [2.]	$\mathscr{M}$-invariant, with respect to $E$, if:
		\begin{enumerate}
			\item [a)] $P$ is  $C^\infty$-invariant;
			
			\item [b)] $P(\mathscr{E}_{\mathscr{M}}(X)) \subseteq \mathscr{E}_{\mathscr{M}}(X)$;
			
			\item [c)] $P: \mathscr{E}_{\mathscr{M}}(X)\to \mathscr{E}_{\mathscr{M}}(X)$  has continous extension to $\mathscr{E}_{\mathscr{M}}'(X)$.
		\end{enumerate}

		\item [3.]	$(\mathscr{M})$-invariant, with respect to $E$, if:
		\begin{enumerate}
			\item [a)] $P$ is  $C^\infty$-invariant;
			
			\item [b)] $P(\mathscr{E}_{(\mathscr{M})}(X)) \subseteq \mathscr{E}_{(\mathscr{M})}(X)$;
			
			\item [c)] $P: \mathscr{E}_{(\mathscr{M})}(X)\to \mathscr{E}_{(\mathscr{M})}(X)$  has continous extension to $\mathscr{E}_{(\mathscr{M})}'(X)$.
		\end{enumerate}
	\end{enumerate}
	
\end{defn}

Note that if  $P$ is  $\mathscr{M}$-invariant  (($\mathscr{M})$-invariant), then there exists a sequence of matrices 
$\sigma_P(\ell)\in \mathbb{C}^{d_\ell\times d_\ell}$,  $\ell\in\mathbb{N}_0$, such that
\begin{equation}\label{Fourier_dist}
	\widehat{P u}(\ell)=\sigma_P(\ell)\widehat{u}(\ell),\ \forall u\in  \mathscr{E}_{\mathscr{M}}'(X) \ (\forall u\in  \mathscr{E}_{(\mathscr{M})}'(X)),
\end{equation}
implying
\begin{equation}\label{action_P_on_u}
	Pu=\sum_{\ell=0}^\infty\sum_{m=1}^{d_\ell}\widehat{Pu}(\ell,m)
	e_{\ell,m} = \sum_{\ell=0}^\infty\sum_{m=1}^{d_\ell}(\sigma_P(\ell)\widehat{u}(\ell))_me_{\ell,m}.	
\end{equation}
We say that $\{\sigma_P(\ell)\}_{\ell \in \mathbb{N}_0}$ is the \textit{matrix-symbol} of $P$. 

Finally, in  order to connect the global hypoellipticity of an invariant operator $P$ to the behavior of its matrix-symbol, we introduce the following number:
\begin{align*}
	m(\sigma_P(\ell)) & \doteq \inf \left\{ \|\sigma_P(\ell)v\|_\mathtt{HS}; \ v\in \mathbb{C}^{d_\ell} \ \textrm{ and } \  \|v\|_\mathtt{HS}=1 \right\}, \ \ell \in \mathbb{N}_0.
\end{align*}

\section{Global $\mathscr{M}$-hypoellipticity \label{sec3}}

In this section, we analyze the  global hypoellipticity problem in the Roumieu case. 

\begin{defn}
	We say that a linear operator $P:\mathscr{E}_{\mathscr{M}}'(X) \to \mathscr{E}_{\mathscr{M}}'(X)$ is globally $\mathscr{M}$-hypoelliptic if conditions
	$$
	u \in \mathscr{E}_{\mathscr{M}}'(X)   \ \textrm{ and } \
	Pu \in \mathscr{E}_{\mathscr{M}}(X) 
	$$
	imply $u \in \mathscr{E}_{\mathscr{M}}(X)$. 
\end{defn}

\begin{thm}\label{theorem-GH-R}
	An $\mathscr{M}$-invariant operator $P$ is globally $\mathscr{M}$-hypoelliptic if and only if for every $\epsilon >0$ there exists $C_\epsilon>0$ such that
	\begin{equation} \label{dio_cond_Roumieu}
		m(\sigma_P(\ell))\geq \exp\left(-\mathcal{M}(\epsilon\lambda_\ell^{1/\nu})\right), \ \forall \ell \geq C_\epsilon.	
	\end{equation}
\end{thm}

\begin{proof}
	Let us start with the sufficiency. For this, assume condition \eqref{dio_cond_Roumieu}  and let
	$u \in \mathscr{E}_{\mathscr{M}}'(X)$ be a solution of $Pu = \phi \in \mathscr{E}_{\mathscr{M}}(X)$. By \eqref{Fourier_dist} and  \eqref{action_P_on_u} we have 
	$$
	\sigma_P(\ell)\widehat{u}(\ell) = \widehat{\phi}(\ell),  \ \forall \ell \in \mathbb{N}_0.
	$$
	
	Since $\phi \in \mathscr{E}_{\mathscr{M}}(X)$, there are 
	constants $C>0$ and $L_\phi>0$  such that
	$$
	\|\widehat{\phi}(\ell)\|_{\mathtt{HS}} \leq C\exp\left(-\mathcal{M}(L_\phi\lambda_\ell^{1/\nu})\right), \ \forall \ell \in \mathbb{N}_0,
	$$
	and setting  $\widetilde{L}_\phi \doteq \frac{L_\phi}{\sqrt{A}H}$ we obtain from Lemma \ref{lemma_several_estimates} that 
	\begin{equation*}
		\|\widehat{\phi}(\ell)\|_{\mathtt{HS}} \leq 
		C\exp\left(-\mathcal{M}(\widetilde{L}_\phi\lambda_\ell^{1/\nu})\right)\exp\left(-\mathcal{M}(\widetilde{L}_\phi\lambda_\ell^{1/\nu})\right),	 \ \forall \ell \in \mathbb{N}_0.
	\end{equation*}
	
	Let $\epsilon \doteq \widetilde{L}_\phi$  and consider $C_{\epsilon}>0$  satisfying 	\eqref{dio_cond_Roumieu}. Then,  $m(\sigma_P(\ell))\neq 0$ for all $\ell \geq C_\epsilon$, implying  
	$$
	\widehat{u}(\ell)=\sigma_P(\ell)^{-1}\widehat{\phi}(\ell), \ \forall \ell \geq C_\epsilon.
	$$
	Hence, by identity $\|\sigma_P(\ell)^{-1}\|_{\mathcal{L}(\mathbb{C}^{d_{\ell}})} = m(\sigma_P(\ell))^{-1}$, we obtain
	\begin{align*}
		\|\widehat{u}(\ell)\|_{\mathtt{HS}}  &
		\leq \exp\left[\mathcal{M}(\widetilde{L}_\phi\lambda_{\ell}^{1/\nu})\right]
		\|\widehat{\phi}(\ell)\|_{\mathtt{HS}}  \\
		& \leq C \exp\left[-\mathcal{M}(\widetilde{L}_\phi \lambda_{\ell}^{1/\nu})\right],	
	\end{align*}
	for all $\ell \geq C_{\epsilon}$. Then, $u \in \mathscr{E}_{\mathscr{M}}(X)$ and the sufficiency is proved.

	For the  necessary part we proceed by a contradiction argument: we  assume that \eqref{dio_cond_Roumieu} fails and exhibit 
	$u \in \mathscr{E}_{\mathscr{M}}'(X) \setminus \mathscr{E}_{\mathscr{M}}(X)$ so that 
	$Pu \in \mathscr{E}_{\mathscr{M}}(X)$.
	
	Under this assumption, there exists $\epsilon_0>0$ satisfying the following: for every $C>0$ there is $\ell>C$ such that
	$$
	m(\sigma_P(\ell))< \exp\left(-\mathcal{M}(\epsilon_0\lambda_\ell^{1/\nu})\right).
	$$
	
	In particular, for $C=1$ we obtain $j_1>1$ such that 
	$$
	m(\sigma_P(\ell_1))< \exp\left(-\mathcal{M}(\epsilon_0\lambda_{\ell_1}^{1/\nu})\right),
	$$
	and $v_{\ell_1}\in\mathbb{C}^{d_{\ell_1}}$  satisfying 
	$\|v_{\ell_1}\|_{\mathtt{HS}}=1$ and 
	$$
	\|\sigma_P(\ell_1)v_{\ell_1}\|_{\mathtt{HS}}<\exp\left(-\mathcal{M}(\epsilon_0\lambda_{\ell_1}^{1/\nu})\right).
	$$
	Hence, inductively, we obtain $\{v_{\ell_k}\}_{k\in\mathbb{N}}$ 
	such that  $\|v_{\ell_k}\|_\mathtt{HS}=1$  and 
	
	\begin{equation}\label{3.4 do art wag}
		\|\sigma_P(\ell_k)v_{\ell_k}\|_\mathtt{HS}
		<\exp\left(-\mathcal{M}(\epsilon_0\lambda_{\ell_k}^{1/\nu})\right), \ \forall k \in \mathbb{N}.
	\end{equation}

	Now, we set 
	$$
	\widehat{u}(\ell)\doteq\left\{
	\begin{array}{l}
		v_{\ell_k}, \ \text{ if } \ \ell=\ell_k, \\ 
		0, \  \ \text{ if } \  \ell\neq \ell_k,\ 
	\end{array}
	\right.
	$$
	and 
	$u\doteq\sum_{\ell=0}^\infty \sum_{m=1}^{d_\ell}\widehat{u}(\ell,m)e_{\ell,m}.$
	
	Let us to show that $u \in \mathscr{E}_{\mathscr{M}}'(X) \setminus \mathscr{E}_{\mathscr{M}}(X)$. Given $L>0$, from  the fact that $\mathcal{M}$ is non-decreasing  and 
	$0=\lambda_0<\lambda_k<\lambda_{k+1}$,  $k\in\mathbb{N}$, we get  
	$$
	\|\widehat{u}(\ell)\|_\mathtt{HS} \leq 1 < 	\exp\left(\mathcal{M}(L\lambda_\ell^{1/\nu})\right), \ \forall \ell\in\mathbb{N}_0.
	$$
	Thus, $u\in\mathscr{E}_\mathscr{M}^\prime(X)$ in view of Theorem \ref{T_charac_ultradistribution}. Moreover, since $\|\widehat{u}(j_k)\|_\mathtt{HS}=1$, for all $k \in \mathbb{N}$, it follows from \eqref{smooth-charac} that $u \notin C^\infty(X)$ and consequently $u \notin \mathscr{E}_{\mathscr{M}}(X)$.
	
	Finally, note that 
	$$
	Pu=\sum_{k=1}^\infty\sum_{r=1}^{d_{\ell_k}}
	\left(\sigma_P(\ell_k)\widehat{u}(\ell_k)\right)_re_{\ell_k,r},
	$$
	hence, by \eqref{3.4 do art wag}, 
	$$
	\|\widehat{Pu}(j_k)\|_\mathtt{HS}=\|\sigma_P(\ell_k)v_{\ell_k}\|_\mathtt{HS} < \exp\left(-\mathcal{M}(\epsilon_0\lambda_{\ell_k}^{1/\nu})\right), \ \forall k\in\mathbb{N}.		
	$$
	
	Choosing $C=1$ and  $L=\epsilon_0$, it follows from Theorem \ref{T_charac_ultrad} that $Pu$ belongs to  $\mathscr{E}_{\mathscr{M}}(X)$, implying that $P$ is not globally $\mathscr{M}$-hypoelliptic.
	
\end{proof}

\begin{rem}
	Note that if we pick $\mathscr{M} = \{(k!)^s\}_{k \in \mathbb{N}}$, $1 < s < \infty$, we may see $\mathscr{E}_{\mathscr{M}}(X)$
	as the Gevrey class $\mathscr{G}^s(X)$ on $X$, since  $\mathcal{M}(r)\sim r^{1/s}$, as $r \to \infty$. Therefore,  a $\mathscr{G}^s$-invariant operator $P$ is globally $\mathscr{G}^s$-hypoelliptic if and only if for every $\epsilon >0$ there exists $C_\epsilon>0$ such that
	\begin{equation*}
		m(\sigma_P(\ell))\geq\exp\left(- \epsilon\lambda_\ell^{1/s\nu}\right), \ \forall \ell \geq C_\epsilon.	
	\end{equation*}
\end{rem}

\begin{rem}
	We recall that A. Kirilov and W. A. A. de Moraes stdudied 
	the global hypoellipticity in the $C^{\infty}$ sense, namely, the problem
	$$
	u \in \mathcal{D}'(X)   \ \textrm{ and } \
	Pu \in C^{\infty}(X) \ \Longrightarrow \ u \in C^{\infty}(X).
	$$
	They proved that a $C^\infty$-invariant operator $P$ is  globally $C^{\infty}$-hypoelliptic if and only if there are postive constants $L, M$ and $R$ such that
	$$
	m(\sigma_P(\ell)) \geq L (1+ \lambda_{\ell})^{M/\nu}, \ \forall \ell \geq R.
	$$

	Notice that given $\epsilon>0$ there is  $C_{\epsilon}>0$ such that
	$$
	(1+ \lambda_{\ell})^{M/\nu} \geq
	\exp\left(-\mathcal{M}(\epsilon \lambda_{\ell}^{1/\nu})\right), \forall \ell \geq C_{\epsilon},
	$$
	since $\lambda_{\ell} \to \infty$ and $\mathcal{M}$ is a positive non-decreasing function. By this discussion we have the following: \textit{if $P$ is  globally $C^{\infty}$-hypoelliptic and  $\mathscr{M}$-invariant, then it is also globally $\mathscr{M}$-hypoelliptic.}
\end{rem}

\section{Global ($\mathscr{M}$)-hypoellipticity \label{sec4}}

We now consider  the global hypoellipticity problem in the Beurling case. 

\begin{defn}
	We say that a linear operator $P:\mathscr{E}_{(\mathscr{M})}'(X) \to \mathscr{E}_{(\mathscr{M})}'(X)$ is globally $(\mathscr{M})$-hypoelliptic if conditions
	$$
	u \in \mathscr{E}_{(\mathscr{M})}'(X)   \ \textrm{ and } \
	Pu \in \mathscr{E}_{(\mathscr{M})}(X)
	$$
	imply $u \in \mathscr{E}_{(\mathscr{M})}(X)$. 
\end{defn}

\begin{thm}\label{theorem-GH-B}
	An  $(\mathscr{M})$-invariant operator $P$ is globally $(\mathscr{M})$-hypoelliptic if and only if there are positive constants $K$, $r$ and $C$ such that
	\begin{equation} \label{dio_cond_Beurling}
		m(\sigma_P(\ell))\geq K\exp\left(-\mathcal{M}(r\lambda_\ell^{1/\nu})\right), \ \textrm{ for all } \  \ell\geq C.
	\end{equation}
\end{thm}

\begin{proof}
	Consider $u \in \mathscr{E}_{(\mathscr{M})}'(X)$ such that $Pu = \phi \in \mathscr{E}_{\mathscr{M}}'(X)$ and assume condition \eqref{dio_cond_Beurling}. In this case,  we have 
	$\widehat{u}(\ell)=\sigma_P(\ell)^{-1}\widehat{\phi}(\ell)$,
	and 
	\begin{equation*}
		\|\widehat{u}(\ell)\|_\mathtt{HS} \leq
		\frac{1}{K}
		\exp\left(\mathcal{M}(r\lambda_{\ell}^{1/\nu})\right)
		\|\widehat{\phi}(\ell)\|_\mathtt{HS}, \  \forall \ell \geq C.
	\end{equation*}
	Once $\phi \in \mathscr{E}_{(\mathscr{M})}(X)$ we obtain for every $L'>0$ a constant  $C_{L'}>0$ such that
	\begin{equation}\label{ult res 2}	\|\widehat{u}(\ell)\|_\mathtt{HS}\leq \frac{C_{L^\prime}}{K}\exp\left(\mathcal{M}(r\lambda_{\ell}^{1/\nu})
		\right) \exp\left(-\mathcal{M}(L'\lambda_{\ell}^{1/\nu})\right).
	\end{equation}
	
	Consider $L>0$ be fixed.	If $r\leq L$,  we have
	$$
	\exp\left(\mathcal{M}(r\lambda_{\ell}^{1/\nu})\right) \leq \exp\left(\mathcal{M}(L\lambda_{\ell}^{1/\nu})\right), \ \forall \ell \in \mathbb{N},$$ 
	since $\mathcal{M}$ is non-decreasing. Therefore, by setting $L^\prime\doteq L\sqrt{A}H$, we obtain from Lemma \ref{lemma_several_estimates}	that		
	$$
	\exp\left(-\mathcal{M}(L\sqrt{A}H\lambda_{\ell}^{1/\nu})\right)
	\leq \exp\left(-2\mathcal{M}(L\lambda_{\ell}^{1/\nu})\right), \ \forall \ell \in \mathbb{N},
	$$	
	and by \eqref{ult res 2} we get	
	$$
	\|\widehat{u}(\ell)\|_\mathtt{HS}  \leq 
	\frac{C_L}{K}\exp\left(-\mathcal{M}(L\lambda_\ell^{1/\nu})\right), \ \forall \ell \geq C.
	$$

	On the other hand, if $r > L$, a new application of Lemma \ref{lemma_several_estimates}, with  $L^\prime\doteq r\sqrt{A}H$, gives 
	$$
	\exp\left(-\mathcal{M}(r\sqrt{A}H\lambda_{\ell}^{1/\nu})\right)\leq
	\exp\left(-2\mathcal{M}(r\lambda_{\ell}^{1/\nu})\right), \ \forall \ell \in \mathbb{N}.
	$$

	Hence, for any $L>0$ there exists $C_L>0$ such that
	$$
	\|\widehat{u}(\ell)\|_\mathtt{HS}\leq C_L\exp\left(-\mathcal{M}(L\lambda_{\ell}^{1/\nu})\right), \ \forall \ell \geq C,
	$$
	implying $\phi \in \mathscr{E}_{(\mathscr{M})}(X)$. Therefore, we have proved the sufficiency of \eqref{dio_cond_Beurling}.
	
	\bigskip 
	
	Conversely, let us assume that \eqref{dio_cond_Beurling} fails. In this case, for every positive constants $K$, $r$ and $C$ there is $\ell>C$ such that
	$$
	m(\sigma_P(\ell))<K \exp\left(-\mathcal{M}(r\lambda_{\ell}^{1/nu})\right).
	$$

	Then, by an inductive argument, we may construct  a sequence
	$\{v_{\ell_k}\}_{k \in \mathbb{N}}$ such that 
	$v_{\ell_k}\in \mathbb{C}^{d_{\ell_k}}$, $\|v_{\ell_k}\|_\mathtt{HS}=1$ and
	\begin{equation*}	
		\|\sigma_P(\ell_k)v_{\ell_k}\|_\mathtt{HS}< \exp\left(-\mathcal{M}(k\lambda_{\ell_k}^{1/\nu})\right),\ \forall  k\in \mathbb{N}.
	\end{equation*}
	
	Now, consider
	$$
	\widehat{u}(\ell) \doteq \left \{
	\begin{array}{l}
		v_{\ell_k} \ \textrm{ if } \ \ell=\ell_k, \\ 
		0, \ \textrm{ if } \ \ell \neq \ell_k,
	\end{array}
	\right.
	$$
	and 
	$u \doteq \sum_{\ell=0}^\infty \sum_{m=1}^{d_\ell}\widehat{u}(\ell,m)e_{\ell,m}.$

	We claim that $u\in\mathscr{E}_{(\mathscr{M})}^\prime(X)\setminus\mathscr{E}_{(\mathscr{M})}(X)$. Indeed, 
	\begin{align*}
		\|\widehat{u}(\ell)\|_\mathtt{HS} 
		& \leq \exp\left(-\mathcal{M}(\lambda_\ell^{1/\nu})\right)\exp\left(\mathcal{M}(\lambda_\ell^{1/\nu})\right) \\
		& \leq 
		\exp\left(-\mathcal{M}(\lambda_0^{1/\nu})\right)\exp\left(\mathcal{M}(\lambda_\ell^{1/\nu})\right) \\
		& = \exp\left(\mathcal{M}(\lambda_\ell^{1/\nu})\right), 
	\end{align*}
	for all $\ell \in \mathbb{N}$, since $\lambda_0=0$. Once $\|\widehat{u}(\ell_k)\|_\mathtt{HS}=1$,  $\forall k \in \mathbb{N}$, it follows that  
	$u \in \mathscr{E}_{(\mathscr{M})}^\prime(X)\setminus  \mathscr{E}_{(\mathscr{M})}(X)$.

	Finally, we proof that $Pu \in \mathscr{E}_{(\mathscr{M})}(X)$. Since
	\begin{equation*}
		Pu=
		\sum_{k=1}^\infty\sum_{m=1}^{d_{\ell_k}}
		(\sigma_P(\ell_k)\widehat{u}(\ell_k))_m e_{\ell_k,m},	
	\end{equation*}
	we have
	$$
	\|\widehat{Pu}(\ell_k)\|_\mathtt{HS}=
	\|\sigma_P(\ell_k)v_{\ell_k}\|_\mathtt{HS}
	<\exp\left(-\mathcal{M}(k\lambda_{\ell_k}^{1/\nu})\right), \ \forall k\in \mathbb{N}.
	$$
	
	Consider $L>0$ be fixed. For all  $k \geq L$,
	$$
	\|\widehat{Pu}(\ell_k)\|_\mathtt{HS}< \exp\left(-\mathcal{M}(k\lambda_{\ell_k}^{1/\nu})\right)\leq
	\exp\left(-\mathcal{M}(L\lambda_{\ell_k}^{1/\nu})\right),
	$$	
	while if $k<L$ we may obtain $C_{L,k}>0$  such that
	\begin{equation}\label{max}
		\exp\left(-\mathcal{M}(k\lambda_{\ell_k}^{1/\nu})\right)
		\leq C_{L,k}\exp\left(-\mathcal{M}(L\lambda_{\ell_k}^{1/\nu})\right),
	\end{equation}
	implying
	$$
	\|\widehat{Pu}(\ell_k)\|_\mathtt{HS}\leq  C_L\exp\left(-\mathcal{M}(L\lambda_{\ell_k}^{1/\nu})\right),\ k<L,
	$$
	where $C_L$ is the max of all constants $C_{L,k}$ satisfying \eqref{max}, for $1\leq k<L$.

	Therefore, it follows from   Theorem  \ref{T_charac_ultrad}  that  $Pu \in \mathscr{E}_{(\mathscr{M})}(X)$. This proves that $P$ is not globally $(\mathscr{M})$-hypoelliptic and the necessity of \eqref{dio_cond_Beurling}.

\end{proof}

\bibliographystyle{plain}
\bibliography{references}

\end{document}